\documentclass{colt2020}

\usepackage{times}

\usepackage{xspace,exscale,relsize}
\usepackage{fancybox,shadow}

\usepackage{color}
\usepackage{amsfonts}

\usepackage[title]{appendix}

\makeatletter
\let\over=\@@over \let\overwithdelims=\@@overwithdelims
\let\atop=\@@atop \let\atopwithdelims=\@@atopwithdelims
\let\above=\@@above \let\abovewithdelims=\@@abovewithdelims
\makeatother
\usepackage{rotating}

\usepackage{ifpdf}

\usepackage{psfrag}

\usepackage{prettyref,enumerate}

\usepackage{tikz}
\usetikzlibrary{arrows}
\tikzstyle{int}=[draw, fill=blue!20, minimum size=2em]
\tikzstyle{dot}=[circle, draw, fill=blue!20, minimum size=2em]
\tikzstyle{init} = [pin edge={to-,thin,black}]

\usepackage[all]{xy}

\usepackage{mathtools}

\newcommand{\mreals}{\ensuremath{\mathbb{R}}}

\ifx\eqref\undefined
\newcommand{\eqref}[1]{~(\ref{#1})}
\fi
\ifx\mod\undefined
\def\mod{\mathop{\rm mod}}
\fi

\usepackage{bm}

\def\argmin{\mathop{\rm argmin}}

\def\exp{\mathop{\rm exp}}

\def\EE{\Expect}

\def\Var{\mathrm{Var}}

\def\PP{\mathbb{P}}

\def\eqdef{\triangleq}

\def\simiid{\stackrel{iid}{\sim}}

\newcommand{\stepa}[1]{\overset{\rm (a)}{#1}}
\newcommand{\stepb}[1]{\overset{\rm (b)}{#1}}
\newcommand{\stepc}[1]{\overset{\rm (c)}{#1}}
\newcommand{\stepd}[1]{\overset{\rm (d)}{#1}}

\newcommand{\Poi}{\mathrm{Poi}}

\newcommand{\naturals}{\mathbb{N}}
\newcommand{\integers}{\mathbb{Z}}
\newcommand{\complex}{\mathbb{C}}
\newcommand{\Expect}{\mathbb{E}}
\newcommand{\expect}[1]{\mathbb{E}\left[#1\right]}
\newcommand{\Prob}{\mathbb{P}}
\newcommand{\prob}[1]{\mathbb{P}\left[#1\right]}
\newcommand{\pprob}[1]{\mathbb{P}[#1]}

\newcommand{\TV}{{\rm TV}}

\newcommand{\diff}{{\rm d}}

\newcommand{\iid}{i.i.d.\xspace}
\newcommand{\ind}{ind.\xspace}

\newcommand{\fracd}[2]{\frac{\diff #1}{\diff #2}}

\newcommand{\pth}[1]{\left( #1 \right)}
\newcommand{\qth}[1]{\left[ #1 \right]}
\newcommand{\sth}[1]{\left\{ #1 \right\}}

\newcommand{\iiddistr}{{\stackrel{\text{\iid}}{\sim}}}
\newcommand{\inddistr}{{\stackrel{\text{\ind}}{\sim}}}

\newcommand{\Binom}{\text{Binom}}
\newcommand{\Bino}{\text{Binom}}

\newcommand{\indc}[1]{{\mathbf{1}_{\left\{{#1}\right\}}}}

\definecolor{myblue}{rgb}{.8, .8, 1}
\definecolor{mathblue}{rgb}{0.2472, 0.24, 0.6} 
\definecolor{mathred}{rgb}{0.6, 0.24, 0.442893}
\definecolor{mathyellow}{rgb}{0.6, 0.547014, 0.24}

\newcommand{\calN}{{\mathcal{N}}}

\newcommand{\calP}{{\mathcal{P}}}

\newcommand{\calT}{{\mathcal{T}}}

\newcommand{\calX}{{\mathcal{X}}}

\newrefformat{eq}{(\ref{#1})}
\newrefformat{thm}{Theorem~\ref{#1}}
\newrefformat{th}{Theorem~\ref{#1}}
\newrefformat{chap}{Chapter~\ref{#1}}
\newrefformat{sec}{Section~\ref{#1}}
\newrefformat{algo}{Algorithm~\ref{#1}}
\newrefformat{fig}{Fig.~\ref{#1}}
\newrefformat{tab}{Table~\ref{#1}}
\newrefformat{rmk}{Remark~\ref{#1}}
\newrefformat{clm}{Claim~\ref{#1}}
\newrefformat{def}{Definition~\ref{#1}}
\newrefformat{cor}{Corollary~\ref{#1}}
\newrefformat{lmm}{Lemma~\ref{#1}}
\newrefformat{prop}{Proposition~\ref{#1}}
\newrefformat{pr}{Proposition~\ref{#1}}
\newrefformat{app}{Appendix~\ref{#1}}
\newrefformat{apx}{Appendix~\ref{#1}}
\newrefformat{ex}{Example~\ref{#1}}
\newrefformat{exer}{Exercise~\ref{#1}}
\newrefformat{soln}{Solution~\ref{#1}}
\newrefformat{ineq}{inequality~(\ref{#1})}

\def\unifto{\mathop{{\mskip 3mu plus 2mu minus 1mu%
			\setbox0=\hbox{$\mathchar"3221$}%
			\raise.6ex\copy0\kern-\wd0%
			\lower0.5ex\hbox{$\mathchar"3221$}}\mskip 3mu plus 2mu minus 1mu}}

\ifx\lesssim\undefined
\def\simleq{{{\mskip 3mu plus 2mu minus 1mu%
			\setbox0=\hbox{$\mathchar"013C$}%
			\raise.2ex\copy0\kern-\wd0%
			\lower0.9ex\hbox{$\mathchar"0218$}}\mskip 3mu plus 2mu minus 1mu}}
\else
\def\simleq{\lesssim}
\fi

\ifx\gtrsim\undefined
\def\simgeq{{{\mskip 3mu plus 2mu minus 1mu%
			\setbox0=\hbox{$\mathchar"013E$}%
			\raise.2ex\copy0\kern-\wd0%
			\lower0.9ex\hbox{$\mathchar"0218$}}\mskip 3mu plus 2mu minus 1mu}}
\else
\def\simgeq{\gtrsim}
\fi

\newif\ifmapx
{\catcode`/=0 \catcode`\\=12/gdef/mkillslash\#1{#1}}
\edef\jobnametmp{\expandafter\string\csname urn_tv2_apx\endcsname}
\edef\jobnameapx{\expandafter\mkillslash\jobnametmp}
\edef\jobnameexpand{\jobname}
\ifx\jobnameexpand\jobnameapx
\mapxtrue
\else
\mapxfalse
\fi

\long\def\apxonly#1{\ifmapx{\color{blue}#1}\fi}
\newcommand{\deltaTV}{\delta_{\TV}}

\title{Extrapolating the profile of a finite population}

\coltauthor{\Name{Soham Jana} \Email{soham.jana@yale.edu} \\ 
	\addr Department of Statistics and Data Science, Yale University, New Haven, CT
\\ \AND 
	\Name{Yury Polyanskiy} \Email{yp@mit.edu} \\
	\addr Department of EECS, MIT, Cambridge, MA  
	\AND 
	\Name{Yihong Wu} \Email{yihong.wu@yale.edu} \\
	\addr Department of Statistics and Data Science, Yale University, New Haven, CT
}

\begin{document}
	\ifpdf
	\DeclareGraphicsExtensions{.pgf}
	\graphicspath{{figures/}{plots/}}
	\fi
	
	\maketitle
	
	\begin{abstract} We study a prototypical problem in empirical Bayes. Namely,
		consider a population consisting of $k$ individuals each belonging to one of $k$ types (some types can be empty).
		Without any structural restrictions, it is impossible to learn the 
		composition of the full population having observed only a small (random) subsample of size $m = o(k)$. 
		Nevertheless, we show that in the sublinear regime of $m =\omega(k/\log k)$, it is possible to consistently estimate in total variation 
		the \emph{profile} of the population, defined as the empirical
		distribution of the sizes of each type, which determines many symmetric properties of the population. We also prove
		that in the linear regime of $m=c k$ for any constant $c$ the optimal rate is $\Theta(1/\log k)$. 
		Our estimator is based on Wolfowitz's minimum distance method, which entails solving a linear program (LP) of size $k$.
		We show that there is a single infinite-dimensional LP whose value simultaneously characterizes the risk of the minimum distance estimator and certifies its minimax optimality.
		The sharp convergence rate is obtained by evaluating this LP using complex-analytic techniques. 	
	\end{abstract}


		\section{Introduction} 
	\label{sec:intro}
	
	Consider a finite population, say, an urn of at most $k$ colored balls, 
	with colors indexed by, without loss of generality, $[k]\triangleq\{1,\ldots,k\}$.
	Let $\theta_j$ denote the the number of balls of color $j\in [k]$ present in the urn.
	We observe a subsample, obtained by revealing each ball independently with probability $p$. This sampling scheme is referred to as the Bernoulli sampling model \cite{BF93}, a specific form of sampling without replacements. 
	We will be interested in both the \emph{linear} and the \emph{sublinear} regime, in which the sampling probability $p$ is a small constant or vanishing as $k$ grows, respectively.
	
	It is not hard to show (see \prettyref{app:mu}) that unless all but a vanishing fraction of the urn is observed, it is impossible to consistently estimate the empirical distribution of the colors, which aligns with the conventional wisdom that the sample size needs to exceed the number of parameters. Fortunately, many interesting properties about the population (such as entropy, number of distinct elements) are label-invariant and hence learnable through the \emph{profile} of the population \cite{orlitsky2005convergence}, 
	defined as the empirical distribution of $\theta=(\theta_1,\ldots,\theta_k)$:
	\begin{align}\label{eq:composition}
	 \pi = {1\over k}\sum_{j=1}^k \delta_{\theta_j}.
	\end{align}
	where $\delta_m$ denotes the Dirac measure (point mass) at $m$,
	Note that $\pi$ is supported on $\{0,\ldots,k\}$ with mean at most one and probability mass function given by $\pi_m = {1\over k} \sum_{j=1}^k \indc{\theta_j=m} $ for $m=0,\dots,k$.
	The profile provides information about the diversity of a population. For example, $\pi=(1-\frac{1}{k})\delta_0+\frac{1}{k}\delta_k$ and $\pi=\delta_1$ correspond to 
the two extremes of all balls having the same color and different colors, respectively.  Furthermore, $\pi_0$ encodes the total number $c$ of distinct colors in urn, since $\pi_{0} = 1-c/k$. 

	Based on the subsampled population, our goal is to reconstruct the profile $\pi$ of the full population. 
	Since many symmetric properties can be expressed as its linear functionals, estimating $\pi$ under the total variation (TV) distance allows simultaneous estimation of all such bounded properties. 
	Our main result is that the profile can be estimated consistently even in the sublinear regime.

	Let $X_j\sim \Bino(\theta_j,p)$ be the number of observed balls of color $j$. The minimax TV risk of estimating $\pi$ is defined as
	\begin{align} 
	R(k) = \inf \sup   \EE[\| \pi - \hat \pi\|_{\TV}]. 
	\label{eq:Rk}
	\end{align}
    where 
		$\| \pi - \hat \pi\|_{\TV} \eqdef \frac{1}{2} \sum_{m\geq 0} |\pi_{m} - \hat \pi_{m}|$, 
		the supremum is over all urns of at most $k$ balls, and the infimum is over all estimators $\hat{\pi}$ as a function of $X=(X_1,\ldots,X_k)$. Our main result is the following.
	
	\begin{theorem}\label{thm:main} 
		There exist absolute constants $c,C,d_0$, such that if $\log k\geq {d_0 \over \bar p}$, then
		\begin{equation}\label{eq:main}
		\min\sth{{\bar p \over p},\sqrt{\log k}}{c\over \log k} \leq R(k) \le \min \sth{{C\over p \log k},1},
		\end{equation}	
		 where $\bar p=1-p$. Furthermore, the upper bound in fact holds for all $p\in(0,1)$, achieved by a minimum-distance estimator computable in polynomial time.
	\end{theorem}
	
	In the linear regime, \prettyref{thm:main} shows that the optimal TV rate is $\Theta(\frac{1}{\log k})$ for any constant sampling probability $p$. 
	This should be contrasted with the estimation of $\pi_0$, known as the distinct elements problem, which has been extensive studied in the literature \cite{BF93,CCMN00,RRSS09,VV11,WY16-distinct}.
	The precise behavior of the minimax risk of estimating $\pi_0$ was determined in \cite{WY16-distinct}. In particular, if $\frac{1}{\log k} \lesssim p \lesssim 1$, 
	the optimal rate of $\pi_0$ is $k^{-\Theta(p)}$, much faster than estimating $\pi$ itself.
	Our result refines this observation and reveals the following dichotomy:
	the polynomial rate $k^{-\Theta(p)}$ holds not just for estimating $\pi_0$ but  for all $\pi_m$ with $m=o(\log k)$; however, for $m = \Theta(\log k)$, $\pi_m$ is much harder to estimate and the rate is no faster than $\Omega(\frac{1}{(\log k)^2})$. This explains the overall TV risk $\Omega(\frac{1}{\log k})$ for estimating the full distribution $\pi$.

	
	In the sublinear regime, \prettyref{thm:main} shows that consistent estimation is possible if $p = \omega(\frac{1}{\log k})$. Although our current lower bound does not conclude its optimality, it is indeed the case based on existing impossibility results of the distinct element problem that shows $\pi_0$ cannot be estimated with vanishing  error if  $p = O(\frac{1}{\log k})$ \cite{GV-thesis,WY16-distinct}.	
	
	For simplicity, we focus on the Bernoulli sampling model in this paper. The results can be extended to models such as iid sampling or Poisson sampling by the usual simulation or reduction argument (cf.~\cite[Appendix A]{WY16-distinct}).

	\subsection{Related work}
	\label{sec:prior}
	
	While the precise question we are considering here was not studied before, there is a long history of related work. 
	First we observe that the goal of estimating functionals of $\theta=(\theta_1,\ldots,\theta_k)$ is a ``compound statistical decision problem'', in the language
	of~\cite{Robbins51}. 
%
	Instead of studying minimax risks of estimating $\theta$ or its functionals,~\cite{Robbins51} proposed an alternative goal
	(``subminimaxity''), which in our case can be rephrased as follows: construct an estimator which has vanishing excess risk (regret) over that of 
	the oracle estimator $\widehat{k_j}(X_j, \pi)$ having access to empirical distribution $\pi$ of $\theta$.
	The general recipe proposed in~\cite{Robbins51} (and later promulgated by~\cite{Robbins56} under the name of ``empirical
	Bayes''), may roughly be described as a two-step
	procedure: first, one produces an estimate $\hat \pi$ of $\pi$, and then, second, substitutes it into the oracle estimator 
	obtaining $\widehat{k_j}(X_j, \hat \pi)$. Thus, Robbins~\cite{Robbins51} asked (his Problem I) how well can the
	first step be done? Our work addresses this question.


	The main part of our theorem characterizes how well the ``prior'' $\pi$ can be estimated. We mention that while
	empirical Bayes method is sometimes understood only as a way to derive estimates of a particular functional of the
	prior, as, for example, in the Good-Turing estimator for the number of unseen species, the idea of estimating the prior
	itself has also been proposed in~\cite{Robbins56,edelman1988estimation}. Furthermore, the solution advocated therein, Wolfowitz's \emph{minimum
	distance estimator} \cite{W57}, is the one we employ in the proof of our result.
	In this regard, one of the main contributions of the paper is showing that performance of the minimum distance
	estimators is characterized by means of a certain function $\deltaTV(t)$, defined as the value of an 
	infinite-dimensional {linear program}, which simultaneously can also be used to produce a matching \textit{lower
	bound}. This duality between the upper and the lower bound has previously been observed and operationalized in the context of
	estimating \emph{a single linear functional} in~\cite{JN09,PSW17-colt,PW18-dual2}. Here we extend this program to estimating the \emph{full distribution}, and evaluate the relevant $\deltaTV$ function using complex-analytic techniques.
	
	Arguably, the counterintuitive part of our result is the possibility of estimating the profile $\pi$ consistently in TV, despite the absence of structural assumptions on the urn configuration and despite $p$ possibly vanishing.
	In fact, this is a manifestation of the fascinating effect originally discovered by~\cite{orlitsky2005convergence} and further developed in \cite{valiant2013estimating,han2018local}, namely, although there exists no consistent estimator of the empirical color distribution, its sorted version can be estimated consistently.
	Nevertheless, the best upper bound that can be extracted (see \prettyref{app:sorted} for details) 	from existing results is $O(\frac{1}{\sqrt{\log k}})$ in the linear regime and there is no applicable lower bound. 
	\prettyref{thm:main} shows that this rate is suboptimal by a square root factor, potentially due to the fact that these previous work did not exploit the finiteness of the population.
	
		In terms of techniques, while the approach of \cite{WY16-distinct} to the distinct elements problem relies on polynomial interpolation and approximation, both the scheme (minimum distance estimator) and the lower bound in the present paper involve linear programming (LP), which is more akin in spirit to the work of \cite{VV11,PW18-dual2}. The technical novelty here is that we use tools from complex analysis to analyze the behavior of the LP.	

	Finally, we mention that a different line of research tracing back to \cite{lord1969estimating} studies the ``mirror image'' of our problem: estimating the empirical distribution of parameters 
	$p_1,\ldots,p_k$ 
	from samples $X_j \sim \Bino(\theta, p_j)$. The recent work of \cite{tian2017learning} uses the method of moments to obtain the optimal rate for 
	$\theta=o(\log k)$. 
	This is further improved in \cite{vinayak2019maximum} by analyzing the nonparametric maximum likelihood.
Alas, in this model, even for large population it is not possible to achieve consistent estimation without 	$\theta\to\infty$.
%

	\medskip
	The rest of the paper is organized as follows.
	\prettyref{sec:mind} introduces the minimum distance estimator and a general characterization of its risk by a linear program. 
	Sections \ref{sec:ub} and \ref{sec:lb} are devoted to analyzing the behavior of this LP using complex-analytic techniques and Laguerre polynomials, completing the proof of \prettyref{thm:main}. 
	\prettyref{app:discuss} contains a detailed discussion on related technical results and a list of open problems. 
	Omitted proofs are contained in the rest of the appendices.

	\section{Minimum distance estimator and statistical guarantees}\label{sec:mind}
	As mentioned in the last section, estimation of the profile revolves around the idea of minimum distance method, which fits a statistical model that is closest to the sample distribution with respect to some meaningful statistical distance. Examples of minimum distance estimators can be traced back to as early as \cite{P1900}, which led to the discovery of the famous minimum chi-square method. In the 1950's, Wolfowitz studied minimum distance methods for the first time as a class, for obtaining strongly (almost surely) consistent estimators \cite{W57}. The pioneering work of \cite{B1977} demonstrates how minimum-Hellinger method can improve upon classical estimators such as the maximum likelihood in the presence of outliers. For a comprehensive account and more recent development we refer the readers to the monograph \cite{basu}. 
	
	To describe the paradigm of the minimum distance estimators we first introduce the general setting of Robbins'
	Problem I mentioned in \prettyref{sec:prior}.
%
	Consider a parametric family of distributions $\{P_\theta: \theta\in\Theta\}$ on some measurable space $\calX$, viewed also as a Markov transition kernel $P$ from $\Theta$ to $\calX$.
	Let $d$ be a distance on the space of priors $\calP(\Theta)$.
	Select $\theta_1,\ldots,\theta_k$ from $\Theta$ such that  ${1\over k}\sum_{j=1}^k
	c(\theta_j) \le 1$, where $c:\Theta \to \mreals$ is some cost function (could be zero), resulting in the 
	empirical distribution 
		$ \pi \eqdef {1\over k} \sum_{j=1}^k \delta_{\theta_j} $.
	Given observations $X_j \simiid P_{\theta_j}$, an estimate $\hat \pi(X_1,\ldots,X_k)$ is
	produced with the goal of minimizing $ \EE[d(\hat \pi, \pi)]$. The minimax risk is defined as
		$$ R(k) = \inf_{\hat \pi} \sup_{\theta_1,\ldots,\theta_k} \EE[d(\hat \pi, \pi)]\,.$$
		\vspace{-1.5em}
	\begin{remark} Note that Robbins also defined a related Problem II in which $\theta_j\simiid G$ with 
	$\EE_G[c(\theta)] \le 1$ and the goal is to estimate the prior $G$ instead of
	the (now random) empirical distribution $\pi$. The minimax risk $R_2(k)$ is similarly defined as the supremum
	over all such $G$. We argue that in many cases the difference between $R(k)$ and $R_2(k)$ is insignificant. 

	Indeed, let $\tau_k = \sup_G \EE[d(G,\pi)]$, which due
	to concentration we assume is $o(R(k))$. The comparison $R_2(k) \le R(k) + \tau_k$ is by conditioning on
	$\pi$\apxonly{ (there is a small problem: $\pi$ does not satisfy cost constraint)}. In the opposite direction,
	if, for example, $d(\cdot,\cdot) \le 1$, then $R(k) \le R_2(m) + {m^2\over 2k}$ since by sampling $m$ times from
	$(X_1,\ldots,X_k)$ with replacement we get $m$ samples from Problem 2's setting with $G=\pi$ (except for a set of
	realizations of probability $m^2\over 2k$ on which we drew some $X_j$ multiple times). Applying Problem 2's
	estimator for $m$ samples we get the inequality. In interesting cases, $R_2(k) \asymp R_2(k^\alpha) \ll
	k^{-\beta}$ for any
	$\alpha,\beta>0$, and thus we get $R_1(k) \asymp R_2(k)$.
	\end{remark}
	To solve this problem we proceed by choosing an auxiliary metric $\rho$ on $\mathcal{P}(\Theta)$, the set of probability measures on $\Theta$. Let $\hat \nu =
	{1\over k} \sum_{j=1}^k \delta_{X_j}$ be the empirical distribution of the sample. 
	Note that in expectation we have, for all $\theta_1,\ldots,\theta_k$,
	\begin{equation}
	\Expect[\hat\nu]=\pi P.
	\label{eq:unbiased}
	\end{equation}	
	where $\pi P = \int P_\theta \pi(d\theta) = {1\over k} \sum_{j=1}^k P_{\theta_j}$.	
	This motivates the following minimum-distance estimator (putting existence of minimum aside):
	\begin{equation}
	\hat \pi = \argmin_{\pi'} \sth{\rho(\hat \nu, \pi'P): \Expect_{\pi'}[c(\theta)]\leq 1}\,. 
	\label{eq:mindist-general}
	\end{equation}	
	To analyze this estimator, suppose, in addition to \prettyref{eq:unbiased}, we have the high-probability guarantee:
	$$ \PP[ \rho(\pi P, \hat \nu) > t_k] \le \epsilon_k \qquad  $$
	for some sequences $t_k,\epsilon_k\to 0$. 
	By the triangle inequality we also have $ \PP[ \rho(\hat \pi P, \pi P) > 2t_k] \le \epsilon_k$.
	Finally, defining the following \emph{deconvolution function}:
	$$ \delta(t) \eqdef \sup \{d(\pi,\pi'): \rho(\pi P, \pi P') \le t, \EE_{\pi}[c(\theta)]\leq 1, \EE_{\pi'}[c(\theta)]
	\le 1 \}\,, $$
	where the supremization is over all distributions $\pi,\pi' \in \calP(\Theta)$. Then we immediately obtain the following high-probability risk bound
	$$ \PP[ d(\hat \pi, \pi) > \delta(2t_k) ] \le \epsilon_k\,.$$
	Using other properties of $d$ and $c$, we can typically convert this into an upper bound for the average risk like
		$$ \EE[d(\hat \pi, \pi)] \lesssim \delta(2t_k)\,.$$
	Selecting different auxiliary metric $\rho$'s results in different estimators. For example, the choice of $\rho$ equal to the
	Kullback-Leibler divergence results in a the non-parametric maximum-likelihood estimator.
	As stated this is all well known. 
	\textit{Our key contribution is the following:} While $\rho$ is
	left arbitrary so far, the choice of $\rho$ being total variation (or Hellinger) distance is special since it comes with an essentially
	matching lower bound. 
	\begin{quote}
	\emph{Meta-principle}. Suppose the loss function $d$
	is of seminorm-type, namely
	$ d(\pi,\pi') = \sup_{T \in \calT} \langle T, \pi-\pi' \rangle $
	for some dual pairing $\langle \cdot, \cdot \rangle$ and a family of linear functionals $\calT$ on
	$\calP(\Theta)$. Take $\rho(\cdot,\cdot) = \|\cdot - \cdot\|_{\TV}$. Then under regularity conditions on $(\Theta, \calX, c, P, \calT)$ we have
	$$ \delta(1/k) \lesssim R(k) \lesssim \delta(t_k) \,.$$
	Thus, when $\delta(1/k) \asymp \delta(t_k)$ we get the sharp rate.
	\end{quote}

	Working out general conditions for the applicability of this program is left for future work. Here we focus
	on the model discussed in the introduction. Recall $\pi=(\pi_{0},\ldots,\pi_k)$ in \prettyref{eq:composition}	denotes the profile of the urn.
	In the Bernoulli sampling model, the observed numbers of balls with color $j$ are independently distributed as 
	\begin{align}
	X_j & \inddistr \Binom(\theta_j, p), \quad j \in [k].
	\end{align}
	Let $\hat \nu = \frac{1}{k} \sum_{j=1}^k \delta_{X_j}$ denote the empirical distribution of the $X_j$'s.
	Then for each $m \geq 0$, we have 
	$\hat\nu_m =\frac{Y_m}{k}$,
	where 
	\begin{equation}\label{eq:ydef}
	Y_m = \sum_{j\in [k]} 1\{X_j = m\}
	\end{equation}
	denotes the number of colors that are observed exactly $m$ times.\footnote{Technically, $\nu_0$ is not directly observed from the sample. Nevertheless, one can compute it by $\hat \nu_{0}\eqdef 1-\sum_{m=1}^k\hat \nu_m$.}
	Define the Markov kernel $P:\integers_+ \to \integers_+$ by $P(i,\cdot) = \Binom(i, p)$, whose transition matrix $P=(P_{im})$ is given by
		\begin{equation}
	P_{im}=\binom{i}{m}p^m(1-p)^{i-m}, \quad i,m \geq 0.
	\label{eq:P}
	\end{equation}	
	Then as in \prettyref{eq:unbiased}, we have the unbiased relation $\Expect[\hat\nu]=\pi P$.
	Particularizing \prettyref{eq:mindist-general} with $\rho=\|\cdot\|_{\TV}$ and $c(\theta)=\theta$, we obtain the following the minimum distance estimator:
	\begin{align}
	\hat \pi = \argmin_{\pi' \in \Pi_k} \|\pi' P - \hat \nu\|_{\TV}\label{eq:pihat}
	\end{align}
	where 
	\begin{equation}
	\Pi_k\eqdef\sth{\pi'\in \calP\{0,1,\dots,k\}: \sum_{m=0}^{k} m \pi_{m}'\leq 1 },
	\label{eq:Pik}
	\end{equation}	
	with $ \calP\{0,1,\dots,k\}$ being the set of all probability mass functions on $\{0,1,\dots,k\}$. As mentioned in \prettyref{sec:intro}, the true profile $\pi$ belongs to $\Pi_k$.
	The estimator \prettyref{eq:pihat} is an LP with $k+1$ variables and can be solved in time that is polynomial in $k$. 
	We will show that it attains the minimax upper bound in \prettyref{thm:main}. As the first step, we relate the minimax risk $R(k)$ to the following LP of modulus of continuity type: for each $0<t<1$,
	\begin{align}
	\deltaTV(t) \triangleq \sup \{\|\pi - \pi'\|_{\TV}: \|\pi P - \pi' P\|_{\TV} \le t;\ \pi,\pi'\in \Pi \}, \label{eq:linear prog}
	\end{align}
	where $\Pi \triangleq \Pi_\infty$ as in \prettyref{eq:Pik}, that is, the set of all distributions on $\integers_+$ with mean at most one.	
The following result shows that the value of this LP characterizes the minimax risk.
	
	\begin{theorem} 
		\label{thm:risk-deltatv}
		There exist absolute constants $C_1,C_2, d_0$ such that for all $k\geq d_0$
		\begin{equation}
		\frac 1{72}\deltaTV\pth{\frac 1{6k}}-{C_2\over \sqrt k} \le R(k) \le 2 \deltaTV\pth{\sqrt{\frac{C_1 \log k}{k}}},
		\label{eq:risk-deltatv}
		\end{equation}
		where the upper bound is attained by the minimum distance estimator given in \prettyref{eq:pihat}.
	\end{theorem}
	
	The proof of \prettyref{thm:risk-deltatv} is given in \prettyref{app:risk-deltatv}. 	
	The main idea is as follows. By virtue of the minimum distance estimator $\hat{\pi}$ 
	and the triangle inequality, we have:
		\[
		\|\hat \pi P - \pi P\|_{\TV} \leq 
		\|\hat \pi P - \hat \nu\|_{\TV} + \|\pi P - \hat \nu\|_{\TV} 	\leq 2 \|\pi P - \hat \nu\|_{\TV},
		\]
		which implies that $(\pi,\hat \pi)$ is a feasible pair for $\deltaTV(t)$ with $t=2 \|\pi P - \hat \nu\|_{\TV}$, and hence the following deterministic bound:
		\begin{align}\label{eq:dagger}
		\|\hat \pi - \pi\|_{\TV} \le \deltaTV(2 \|\pi P - \hat \nu\|_{\TV})\,
		\end{align} 
		Recall from \prettyref{eq:unbiased} that $\hat \nu$ is an unbiased estimator of $\pi P$.
		Furthermore, by concentration inequality one can show that with high probability that $\|\hat \nu-\pi P\|_{\TV} = O(\sqrt{\frac{\log k}{k}})$, from which the upper bound quickly follows.
		The lower bound follows from that of estimating linear functionals developed in \cite{PW18-dual2}. Roughly speaking, we use the optimal solution $(\pi,\pi')$ for $\deltaTV(\Theta(1/k))$ to randomly generate two urns of size $\Theta(k)$ whose sampled version are statistically indistinguishable. With appropriate truncation argument, this can be turned into a valid minimax lower bound via Le Cam's method \cite{Tsybakov09}.
		
		%

	\prettyref{thm:risk-deltatv} allows us to reduce the statistical problem \prettyref{eq:Rk} to studying the behavior of $\deltaTV(t)$ for small $t$. This is characterized by the following lemma:
	\begin{lemma}\label{lmm:ratedelta-TV}
		\begin{enumerate}[(1)]
			\item There exists absolute constant $C_3>0$ such that for all $p,t$ we have
			\begin{align} 
			\deltaTV(t)\leq \min \sth{{C_3\over p\log (1/t)},1}.\label{eq:deltaTVub}
			\end{align}
			\item There exist absolute constants $C_4,t_0>0$ such that for any $p\in (0,1)$, $t\leq t_0$, 
			\begin{align}
			\deltaTV(t)\geq \min\sth{{\bar p\over p},\sqrt{\log(1/t)}}{C_4\over \log(1/t)}.\label{eq: deltaTV lb}
			\end{align}
		\end{enumerate}
	\end{lemma} 
	Combining \prettyref{thm:risk-deltatv} and \prettyref{lmm:ratedelta-TV} yields the main result in \prettyref{thm:main}. 
	The next two sections are devoted to the proof of \prettyref{lmm:ratedelta-TV}.
	
\begin{remark}[Reverse data processing]
\label{rmk:reverseDPI}	
	Note that by the data processing inequality (DPI) of TV distance, we have  $\|\pi P-\pi'P\|_{\TV} \leq \|\pi-\pi\|_{\TV}$ and hence
	$\deltaTV(t) \geq t$. 
	Therefore \prettyref{lmm:ratedelta-TV} can be understood as a \emph{reverse DPI} for the binomial kernel $P$ in \prettyref{eq:P}. For example, if $p=\Theta(1)$, then 
	\prettyref{eq:deltaTVub} implies that (which is the best possible in view of \prettyref{eq: deltaTV lb}):
	\[
	\|\pi P-\pi'P\|_{\TV} \geq \exp\sth{-\Theta\pth{\frac{1}{\|\pi-\pi\|_{\TV}^2}}}.
		\]
\end{remark}

	\section{Upper bound on $\deltaTV(t)$ by $H^\infty$-relaxation}
	\label{sec:ub}
	To bound $\deltaTV(t)$ from above, we first relate it to the following LP
	\begin{align}\label{eq:delta*}
	\delta_*(t) \eqdef \sup_{\Delta}\sth{\sum_{m=0}^{\infty} |\Delta_m|: \|\Delta P\|_1 \le t, \sum_{m=0}^\infty m|\Delta_m| \le 1}.
	\end{align}

The next lemma shows how the two LPs \prettyref{eq:linear prog} and \prettyref{eq:delta*}
 are related. The proof is straightforward and deferred till \prettyref{app:proofs}. 
	\begin{lemma}\label{lmm:delta*}
		For all $t\in [0,1]$ we have $\frac{1}{2}(\delta_*(t)-t)\leq \deltaTV(t)\leq \delta_*(t)$.
	\end{lemma}
	\begin{remark}
	Note that our only goal is to substitute estimates on $\deltaTV$ into~\eqref{eq:risk-deltatv}. Therefore, due to the presence of
	the (unavoidable) second term in the LHS of~\eqref{eq:risk-deltatv}, the slight difference between
	$\delta_*(t)-t$ and $\delta_*(t)$ in the lower bound in Lemma~\ref{lmm:delta*} is completely irrelevant and we
	can essentially think of $\deltaTV$ and $\delta_*$ as universally within a factor of two of each other.
	\end{remark}

	\begin{proof}[Proof of upper bound in Lemma~\ref{lmm:ratedelta-TV}]
	We start with recalling a few facts from the complex analysis. Denote the sup-norm of a holomorphic function $f$ over an open set $V\subset\complex$ by
	$\|\cdot\|_{H_{\infty}(V)}$. Let $D=D_1$ be the open unit disk in $\mathbb{C}$ and denote the horodisks for
	$0<p\le 1$ as
	$$ D_p \eqdef \bar p + pD=\{z\in\complex: |z-\bar p|\leq p\}\,.$$
	In addition, we also define another norm for functions analytic in the neighborhood of the origin:
	\begin{equation}\label{eq:anorm}
		\|f\|_A  \eqdef \sum_{j=0}^\infty |a_j|, \qquad f(z) \eqdef \sum_{j\ge 0} a_j z^j\,.
\end{equation}	
		Since $f(re^{i\omega}) \le \sum_{n\ge 0} r^n |a_n| \le \|f\|_A$, we have
		\begin{equation}
		 \|f\|_{H^\infty(D)} \le \|f\|_A\,.
		\label{eq:HAnorm}
		\end{equation}

	In~\cite[(39)]{PSW17-colt} by an application of Hadamard's three-lines theorem, it was shown that for any
	$q\in(0,1)$ and any holomorphic function $f$
	\begin{equation}\label{eq:horo1}
		\|f\|_{H^\infty(D_{1/2})} \le \|f\|_{H^\infty(D)}^{1-2q\over \bar q} \|f\|_{H^\infty(D_q)}^{q\over \bar
	q}\,.
	\end{equation}	
	Indeed, reparametrizing $f(z)=g(\frac{1+z}{1-z})$, we have 
	\begin{equation}
	\|g\|_{H^{\infty}(\Re=r)}=\|f\|_{H^\infty(D_{1/(1+r)})}.
	\label{eq:gfHinfty}
	\end{equation}
	 for $r\geq 0$. Then the Hadamard three-lines theorem applied to $g$ shows that $r\mapsto \log \|f\|_{H^\infty(D_{1/(1+r)})}$
	 is convex, proving \prettyref{eq:horo1}.
	A straightforward generalization (with a different choice of the middle line in the Hadamard theorem) 
	shows that more generally for any $1>q_1 > q>0$ we have
	\begin{equation}\label{eq:dl_ub5}
		\|f\|_{H^\infty(D_{q_1})} \le \|f\|_{H^\infty(D)}^{1-{q \bar q_1\over \bar
	q q_1}} \|f\|_{H^\infty(D_q)}^{q \bar q_1\over \bar
	q q_1}\,.
	\end{equation}
	Next, for any $f$ holomorphic on $\lambda D$ for $\lambda > 0$ we have the following estimate
	\begin{equation}\label{eq:cauchy}
		{1\over \ell !} |f^{(\ell)}(0)| \le \lambda^{-\ell} \|f\|_{H^\infty(\lambda D)}\,.
	\end{equation}
which follows by a Cauchy integral formula: $\frac{f^{(\ell)}(0)}{\ell!}=\frac{1}{2\pi i}\oint_{|z|=\lambda}\frac{f(z)}{z^{\ell+1}}\,dz$.

	With these preparations we move to the proof of~\eqref{eq:deltaTVub}. Consider any sequence $\Delta$ feasible
	for $\delta_*(t)$. For each absolutely summable sequence $\Delta$, we consider its $z$-transform:
	$	f_\Delta(z) \triangleq \sum_{m \geq 0} \Delta_m z^m$, 
	which is a holomorphic function on the open unit disk $D$. 
	Furthermore, using the definition of $P$ in \prettyref{eq:P} and the binomial identity, it is straightforward to verify that
	$f_{\Delta P}  = Pf_\Delta,$
	where the Markov kernel $P$ acts on $f$ as a composition operator 
	$(Pf)(z) \triangleq f(p z+\bar{p})$, where $\bar p\eqdef 1-p$. 
	Given this observation we see that the definition of $\delta_*(t)$ can also be restated as optimization over
	all holomorphic functions on the unit disk, cf.~\eqref{eq:anorm}:
	\begin{equation}\label{eq:deltas_redef}
		\delta_*(t) =  \sup_{f}\sth{\|f\|_A: \|Pf\|_A \le t, \|f'\|_A \le 1}\,.
	\end{equation}	

	For any feasible $f$ in~\eqref{eq:deltas_redef} we have that $\|f'\|_{H^\infty(D)}\le 1$ and
	$\|f\|_{H^\infty(D_p)} \le t$. Thus, integrating $f'$ from some point in $D_p$ we obtain that also
	$ \|f\|_{H^\infty(D)} \le 1+t \le 2\,.$
	Therefore, applying~\eqref{eq:dl_ub5} to $f$ we get
	$$ \|f\|_{H^\infty(D_{3/4})} \le 2 t^{\min(\frac{p}{3\bar p},1)}\,.$$
	Next, since $\frac12D\subset D_{3/4}$ we have from~\eqref{eq:cauchy}
	\begin{equation}\label{eq:dl_ub}
		|\Delta_\ell| ={1\over \ell !} |f^{(\ell)}(0)| \le 2^\ell t^{\min(\frac{p}{3\bar p},1)} \le 2^\ell
	t^{p/3}\,.
	\end{equation}	

	Finally, since for any $\Delta$ feasible for $\delta_*(t)$ we have $\sum_m m|\Delta_m| \le 1$, Markov inequality
	implies $\sum_{m \ge J} |\Delta_m|  \le {1\over J}$ for any integer $J\ge 1$. 
	Together with~\eqref{eq:dl_ub} we conclude that for any feasible $\Delta$-sequence
	\begin{equation}\label{eq:dl_ub3}
			\sum_m |\Delta_m| \le J 2^{J} t^{\frac{p}{3}} + {1\over J}  \le {1\over J} \left(1 + 6^J
		t^{p/3}\right)\,,
	\end{equation}		
	where in the last step we used $J^2 \le 3^J$. Hence, whenever $J\le\left
	\lfloor {p\log {1\over t}\over 3 \log 6} \right \rfloor$, the right-hand side of~\eqref{eq:dl_ub3}
	can be upper-bounded by $2\over J$. This, in view of Lemma~\ref{lmm:delta*} completes the proof of~
	\eqref{eq:deltaTVub} since by definition~$\delta_{\TV}\le 1$. 
	\end{proof}

\begin{remark} Note that functions that saturate~\eqref{eq:horo1} are $f(z) = e^{-m {1+z\over 1-z}}$ where $m \sim \log
{1\over t}$. Computing Taylor coefficients $[z^\ell]f(z)$ of $f(z)$ for $\ell = \Theta(m)$ can be done by applying the saddle-point method to the integral
	$$ [z^\ell] f(z) = {1\over 2\pi i} \oint e^{-m{1+z\over 1-z} - (\ell+1) \log z} dz\,.$$
	It turns out that these coefficients behave in the following way, when $\ell/m = \Theta(1)$:
	$$ [z^\ell] f(z) = \begin{cases} e^{-\Theta(m)}, & \ell/m < 1/2\\
				\Theta\left(1\over \sqrt{m}\right), & \ell/m > 1/2
			\end{cases} $$
	This dichotomy corresponds to critical points of the function ${1+z\over 1-z} - {\ell\over m} \log z$ leaving
	the unit circle when $\ell/m < 1/2$. This shows that the estimate in~\eqref{eq:dl_ub3} is qualitatively tight.
	This effect of sudden jump in the magnitude of coefficients will be the basis of the lower bound in the next section.
\end{remark}

	\section{Lower bound on $\deltaTV(t)$}
	\label{sec:lb}
	
	In view of \prettyref{lmm:delta*} it suffices to consider $\delta_*(t)$ in \eqref{eq:delta*}. Given the
	equivalent definition~\eqref{eq:deltas_redef}, as a warm-up, let us naively replace all $\|\cdot\|_A$ norms with
	$\|\cdot\|_{H^\infty(D)}$. We then get the following optimization problem:
	\begin{align}\label{eq:full relaxation}
	\delta_{H^\infty}(t) &\eqdef \sup\{\|f\|_{H^\infty(D)}:  \|f'\|_{H^{\infty}(D)}\leq 1, \|f\|_{H^{\infty}(\bar{p}+pD)}\le t\}
	\end{align}  
	Note that even though the objective function of \prettyref{eq:full relaxation} is smaller than that of $\delta_*(t)$, the feasible set is also a relaxation.
	Thus $\delta_{H^\infty}(t)$ does not constitute a valid lower bound to $\delta_*(t)$; nevertheless its solution, given in the following lemma, provides important insight on constructing a near-optimal solution for $\delta_*(t)$. 
	\begin{lemma}
	$\delta_{H^\infty}(t)= \Theta_p \pth{1\over \log \pth{1/ t}}$. 
	\end{lemma}
	\begin{proof}
		For the upper bound,  as before we reparameterize $f(z)=g\pth{w}$ with $w={1+z\over 1-z}$. 
		Then \prettyref{eq:gfHinfty} with $r=1/p-1$ implies that $\| g\|_{H^\infty(\Re > \bar{p}/p)} = \| f\|_{H^\infty(\bar{p}+pD)} \le t$.
		By Cauchy's integral formula, we conclude that for some constant $C_{p}$ (here and below possibly
	different on each line) we have 
	$ \|g'\|_{H^\infty(\Re > 2 \bar{p}/p)} \le C_{{p}} t$.
	
	Note that $g'(w) = \frac{2}{(1+w)^2} f'(\frac{w-1}{w+1})$.
	Applying \prettyref{eq:gfHinfty} again with $r=0$ yields $\|g'\|_{H^\infty(\Re > 0)} \leq 2$. 
	Thus from Hadamard's three lines theorem we conclude for any $\epsilon\in(0,\bar p/p)$,
	$ \|g'\|_{H^\infty(\Re = \epsilon)} \le C_{{p}} t^{\min\sth{\epsilon p/(2{\bar{p}}), 1}}$.
	
	Finally, for any $\omega \in \mreals$, integrating the derivative horizontally yields:
	$$ |g(i\omega) - g(i\omega + {\bar{p}}/p)| \le C_{p} \int_0^{{\bar{p}}/p} t^{\epsilon
		p/(2{\bar{p}})} d\epsilon \le C_{p} {1\over \log {1\over t}} $$
	Since $|g(i\omega + {\bar{p}}/p)| \le \| g\|_{H^\infty(\Re = \bar{p}/p)} \leq t$, we conclude that on $\{\Re = 0\}$ we have
	$$ \|g\|_{H^\infty(\Re = 0)} = \|f\|_{H^\infty(D)} \le C_{p} {1\over \log {1\over t}} \,,$$
	proving the upper bound part.	
	
	For the lower bound, consider the following function
	\begin{align}
	f(z) = {c_p\over \log {(1/t)}} (1-z)^2 t^{{p\over \bar{p}} {1+z\over 1-z}}\label{eq:m18}
	\end{align} 
	for some constant $c_p>0$. Then using \eqref{eq:gfHinfty} we have
	$\|f\|_{H^\infty(\bar p+pD)}\leq {4c_p\over \log(1/t)}\sup_{z\in \bar p+pD} |t^{{p\over \bar p}{1+z\over 1-z}}|=
	{4c_p t\over \log(1/t)}$, 
	and 
	\begin{align*}
	\|f'\|_{H^\infty(D)}&= c_p\left\|-{2\over \log(1/t)}(1-z)t^{{p\over \bar p}{1+z\over 1-z}}-{2p\over \bar p}t^{{p\over \bar p}{1+z\over 1-z}}\right\|_{H^\infty(D)}\nonumber \\
	&{\leq} c_p\pth{{4\over \log (1/t)}+{2p\over \bar p}}\left\|t^{{p\over \bar p}{1+z\over 1-z}}\right\|_{H^\infty(D)} 
	\overset{\eqref{eq:gfHinfty}}{=} c_p\pth{{4\over \log (1/t)}+{2p\over \bar p}}
	\leq {2c_p(1+\bar p)\over \bar p}
	\end{align*}
	where the last inequality follows from $\log(1/t)\geq 1$ for all small $t$. 
	This shows $f$ is feasible for $\delta_{H^\infty}(t)$ for small $c_p$. 
	Finally noticing that $\|f\|_{H^{\infty}(D)}\geq |f(-1)|={c_p\over \log(1/t)}$
	concludes the proof.
	\end{proof}
	
	Next we modify \eqref{eq:m18} to produce a feasible solution for $\delta_*(t)$ leading to the following lower bound, 
which, in view of \prettyref{lmm:delta*}, provides the required bound in \eqref{eq: deltaTV lb} on $\deltaTV(t)$.

	\begin{lemma}\label{lmm:delta*-lb}
		There exist absolute constants $C>0$ and $\tilde\beta_0>0$ such that for all $t>0$ and $p\in[0,1)$,
			\begin{align}\label{eq:delta*-lb}
			\delta_*(t)\geq {C\over \tilde \beta}, \quad \tilde\beta \eqdef \max\left( {p\over 1-p} \log {1\over t}, \sqrt{\log {1\over t}\over
			1-p}\right)
			\end{align} 
			provided that $\tilde \beta \ge \tilde \beta_0$.
	\end{lemma}	

	\begin{proof}
		 Fix $p,t\in (0,1)$. Considering~\eqref{eq:deltas_redef} our goal is to find a feasible function
		 and bound its $\|\cdot\|_A$ norm from below. 
		Our main tool for converting between the $\|\cdot\|_A$ norms in the definition~\eqref{eq:deltas_redef} and the
		more convenient $H^\infty$ norms is the following general result complementing \prettyref{eq:HAnorm}: For any $r>1$,
		\begin{equation}\label{eq:atod}
			\|f\|_A \le {1\over \sqrt{1-r^{-2}}} \|f\|_{H^\infty(rD)}\,. 
		\end{equation}		
		Indeed, let $f(z) = \sum_{n\ge 0} a_n z^n$ and let $\tilde f(z) = \sum_{n\ge 0} \tilde a_n z^n$ with
		$\tilde a_n =  a_n r^{n}$ and thus $\tilde f(z) = f(rz)$. From
		the Plancherel identity we have
		$$ \sum_n |\tilde a_n|^2  = {1\over 2\pi} \int_0^{2\pi} |\tilde f(e^{i \omega})|^2 d\omega \le \|\tilde
		f\|_{H^\infty(D)}^2 = \|f\|_{H^\infty(rD)}^2\,.$$
		Thus, \eqref{eq:atod} follows from an application of Cauchy-Schwarz inequality:
		$$ \sum_n |a_n| = \sum_n r^{-n} |\tilde a_n| \le \sqrt{\sum_{n\ge 0} r^{-2n}} \|f\|_{H^\infty(rD)} = 
			 {1\over \sqrt{1-r^{-2}}} \|f\|_{H^\infty(rD)}. $$

		Next, fix some $\beta \ge \beta_0$ and $\tau \in (0,1)$, where $\beta_0\ge 1$ is a numeric constant to be specified later, and let
		$\alpha=1-\tau \in (0,1)$. 
		Consider the function, a modified version of \eqref{eq:m18}, given by 
		\begin{equation}
		h(z)=\tilde h(\alpha z), \qquad \tilde h(z) = \exp\pth{-\beta \frac{1+z}{1-z}}.
		\label{eq:hz}
		\end{equation} 
		Using \prettyref{eq:gfHinfty}, we can explicitly calculate that for any $0<q\le 1$:
		\begin{equation}\label{eq:horo}
			\|\tilde h\|_{H^{\infty}(1-q+qD)} = e^{-\beta {1-q\over q}}\,.
		\end{equation}		

		We will show below the following estimates (all positive numerical constants below, i.e. those that are independent of
		parameters $p,t,\beta$, are denoted by a common symbol $C$):
		\begin{align} 
		   \|h\|_A &\ge C \sqrt{\beta} (1-\tau)^{3\beta\over 2}\label{eq:h2}\\
		   \|h(p\cdot +\bar p)\|_A &\le \tau^{-\frac12} e^{-\beta E},\, \qquad E \eqdef {\bar \tau \bar p \over p+\bar p \tau} \label{eq:h3}\\
		   \|h'\|_A &\le 2 \tau^{-\frac32}\,.\label{eq:h4}
		\end{align}
		
		Thus, taking $f(z) = \frac12 \tau^{\frac32}h(z)$ in~\eqref{eq:deltas_redef} proves that for all $\beta > \beta_0$ we have
		\begin{equation}\label{eq:h5}
			\delta_*\left({\tau\over 2} e^{-\beta E}\right) \ge C \sqrt{\beta \tau^3} (1-\tau)^{3\beta\over
			2}
		\end{equation}
		To show that~\eqref{eq:h5} implies~\eqref{eq:delta*-lb} we set $\tau = {1\over \beta}$ and thus the last
		term in~\eqref{eq:h5} can be lower bounded by $(1-1/\beta_0)^{3\beta_0/2}$ and be absorbed into $C$.
		Notice also that if $\beta\ge 2$ then $\bar \tau \ge 1/2$ and thus $ E \ge {\bar p\over 2} {1\over
		{1\over \beta} + p}$. Since $\tau \le 1$ and $\delta_*$ is monotone in its argument we can simplify
		\begin{equation}\label{eq:h5b}
			\delta_*\left(\exp\left\{-{\beta \over {1\over \beta} + p} {\bar p\over 2} \right\}\right) \ge {C\over \beta} 
		\end{equation}
		Note next that for any $\mu,p>0$, taking $x=\max(\mu p, \sqrt{\mu})$ implies
			$ {x \over {1\over x} + p} \ge {\mu \over 2}$,
			which is verified by considering the two cases $\mu p \lessgtr \sqrt{\mu}$ separately. Then, taking
				$$ \beta = \max(\mu p, \sqrt{\mu}), \qquad \mu \eqdef {4\over \bar p} \log {1\over t} $$
			ensures the argument of $\delta_*$ in~\eqref{eq:h5b} is at most $t$. In summary, we obtain the
			bound~\eqref{eq:delta*-lb} for all $t\le t_0$.

		We proceed to proving~\eqref{eq:h2}-\eqref{eq:h4}. 
		For~\eqref{eq:h3} we set $r={1-\alpha \bar p\over \alpha p}$ in~\eqref{eq:atod} and get
		\begin{align*} \|h(p\cdot +\bar p)\|_A &\le c \|h(p\cdot +\bar p)\|_{H^\infty(rD)}
					= c \|h\|_{H^\infty(\bar p + prD)} 
					= c e^{-\beta {\alpha \bar p\over 1-\alpha \bar p}}\,,
		\end{align*}					
		where we denoted $c = \sqrt{1\over 1-r^{-2}}$ and also applied~\eqref{eq:horo} with $q=\alpha pr =
		1-\alpha \bar p$. We next bound 
		$ c \le {1\over \sqrt{1-r^{-1}}} = \sqrt{1-\alpha \bar p\over 1-\alpha} \le \frac{1}{\sqrt{1-\alpha}}$.

		For~\eqref{eq:h4} we first notice that for any function $f$ holomorphic on $r_2 D$ we can estimate its
		derivative on $r_1 D$, where $r_1<r_2$ via Cauchy integral formula as
		$$ \|f'\|_{H^\infty(r_1D)} \le {1\over r_2 - r_1} \|f\|_{H^\infty(r_2 D)}\,.$$
		Applying this with $f=h$, $r_1 = {1+r_2\over 2}$ and $r_2 = {1\over \alpha}$ we get 
		$$ \|h'\|_{H^\infty(r_2 D)} \le \sqrt{2\over \alpha^{-1}-1} \|h\|_{H^\infty(D/\alpha)} = \sqrt{2\over
		\alpha^{-1}-1}$$
		last step being again via~\eqref{eq:horo} with $q=1$. Applying now~\eqref{eq:atod} with $r=r_2$ we
		obtain overall
		$$ \|h'\|_A \le {2\alpha \over (1-\alpha) \sqrt{1-\alpha^2}} \le {2\over (1-\alpha)^{3/2}}. $$

		To show~\eqref{eq:h2}, we need to analyze the Taylor coefficients of $h$ explicitly as the $H^\infty$-norm bound is too weak. A natural and
		straightforward way is to apply the saddle-point method to study these coefficients. However, due to the special nature of
		$h$ its coefficients have already been well understood. Indeed, in~\cite[5.1.9]{Szego}) it shown that 
		for each $x\in\complex$ and $|v|<1$ 
		\begin{align}
		{e^{-x\frac{v}{1-v}}} = \sum_{n=0}^\infty v^nL_n^{(-1)}(x)\,,\label{eq:generating}
		\end{align}
		where $L_n^{(-1)}(x)$ are \textit{generalized Laguerre polynomial} of degree $n$. We will not need explicit formulae of
		these polynomials and only rely on their asymptotics (of Plancherel-Rotach type),
		cf.~\cite[8.22.9]{Szego}:
		For each $\epsilon>0$ there exists a $C_\epsilon>0$ such that for any $n\ge 0$, any $\epsilon\leq
		\phi\leq \pi/2-\epsilon n^{-1}$, we have 
		\begin{align}
		L_n^{(-1)}(x)&= e^{x/2} (-1)^n(\pi \sin \phi)^{-1/2}x^{1/4}n^{-3/4}\nonumber \\
		&\sth{\sin[n(\sin\ 2\phi -2\phi)+3\pi/4]+(nx)^{-1/2}O_{\epsilon}(1)}, \label{eq:Plancherel}
		\end{align}
		where $x = 4n \cos^2 \phi$ and the $O_\epsilon(1)$ is uniformly bounded by $C_\epsilon$ for all $n$
		and $\phi$.

		Comparing~\eqref{eq:generating} with the definition of $h$ we get
		$h(z)=e^{-\beta} \sum_{m\geq 0} L_m^{-1}(2\beta)z^m\alpha^m$.
		In other words, if we denote the $m$-th coefficient of $h(z)$ by $\Delta_m$, then 
		\begin{equation}\label{eq:h7}
			\Delta_m = e^{-\beta} \alpha^m L_m^{-1}(2\beta)\,.
		\end{equation}
			Due to the oscillatory nature of the Laguerre polynomial, it is not possible to bound $|\Delta_m|$ away from zero.
			Nevertheless, the following lemma shows that two consecutive terms cannot be simultaneously small:
			\begin{lemma}
			\label{lmm:lag}	
		For all $m \in (\beta,3\beta/2)$ and for sufficiently large $\beta$, 
		\begin{equation}\label{eq:h8}
			|\Delta_m| + |\Delta_{m+1}| \ge \alpha^{3\beta/2} \beta^{-1/2} {\sqrt{2}\over 6} \,.
		\end{equation}		
			\end{lemma}
		From here~\eqref{eq:h2} follows simply by
		$ \|h\|_A \ge \sum_{\beta \le m \le 3\beta/2} |\Delta_m| \ge \alpha^{3\beta/2} {\sqrt{2\beta}\over 24} $.
			We note that the estimate \eqref{eq:h2} is tight. Indeed, applying~\eqref{eq:atod} with
		$r={1\over \alpha}$ yields
			$ \|h\|_A \le \frac{1}{\sqrt{1-\alpha^2}} \leq 1/\sqrt{\tau}$,
			where we also used $\|h\|_{H^\infty(D/\alpha)}=\|\tilde h\|_{H^\infty(D)} = 1$
		via~\eqref{eq:horo} with $q=1$.
\end{proof}

%



\acks{The authors thank C. Daskalakis for pointing out~\cite{valiant2013estimating} and G.~Valiant for communicating \cite{GValiant-email_2019}.}

\appendix
\section{Discussions}
	\label{app:discuss}

\subsection{Comparison with previous results}
	\label{app:sorted}

In this section we review previous results on estimating sorted distribution or profile under different loss function and different sampling model. 
	To this end, let us consider an urn with exactly $k$ balls. Then its composition can be described by the distribution $\mu$ on
	$[k]$ with $\mu(x) = \theta_j/k$. When we go from $\mu$ to $\pi$ we erase the ``color labels'' (i.e., if the 
	balls in the urn are arranged as piles of distinct colors, going from $\mu$ to $\pi$ is analogous to turning off the
	lights so that only the heights of each	pile, but not their colors, are shown). This could have been done in a different way by sorting $\mu$. Namely, let us define
	$$ \mu^{\downarrow}_{i} = i\mbox{-th largest atom of $\mu$}. $$
	Note that $\pi$ and $\mu^{\downarrow}$ can be expressed in terms of one another. \apxonly{However, the inverse map (from $\pi$ to
	$\mu^{\downarrow}$) is very ill-conditioned. For example, consider $\mu^1 = \mu^{1\downarrow} = \delta_1$ (all balls of color 1)
	and $\mu^2 = \mu^{2\downarrow} = 1/2 \delta_1 + 1/2 \delta_2$ (half-half of two colors). Then, $\pi^1 = (1-\frac{1}{k})
	\delta_0 + \frac{1}{k} \delta_k, \pi^2 = (1-\frac{2}{k})\delta_0 + \frac{2}{k} \delta_{k/2}$. So we get
	$$ \|\mu^1 - \mu^2\|_{\TV} = \|\mu^{1\downarrow} - \mu^{2\downarrow}\|_{\TV} =  1/2, \|\pi^1 - \pi^2\|_{\TV} =
	\frac{2}{k}\,. $$
	In the opposite direction, we can easily show:} In fact we have
	\begin{equation}\label{eq:mu_pi}
	\|\pi^1 - \pi^2 \|_{\TV} \le 2 \|\mu^{1\downarrow} - \mu^{2\downarrow}\|_{\TV} \le 2 \|\mu^1 - \mu^2\|_{\TV} 
	\end{equation}
	Indeed, the second inequality follows from the fact that decreasing rearrangement minimizes the
	$\ell_1$-distance.\apxonly{(I did it by showing that if $\mu^1$ is sorted then interchanging two atoms of $\mu^2$
	which are not in decreasing order always decreases TV.)} To prove the first inequality, note that 
\begin{equation}\label{eq:tv_w1}
			2 \|\mu^{1\downarrow} - \mu^{2\downarrow}\|_{\TV} = \sum_{j} \left|\sum_{i\ge j} \pi^1_{i} - \pi^2_i \right| =
		W_1(\pi^1, \pi^2)\,.
\end{equation}	
where $W_1$ denotes the 1-Wasserstein distance between probability distributions and, in one dimension, coincides with the $L_1$-distance between the cumulative distribution functions (CDFs). 
	Since $\pi^1,\pi^2$ are supported on $\mathbb{Z}$, the indicator function $1_E$ is $1$-Lipschitz for any
	$E\subset \mathbb{Z}$ and thus $W_1(\pi^1,\pi^2) \ge \|\pi^1 - \pi^2\|_{\TV}$.

	Can one estimate $\mu^{\downarrow}$ from the sample $X$? The answer is yes, in both $\ell_\infty$ and $\ell_1$ (TV),
	as well as other metrics. However, to discuss these results let us move to the setting of Robbins Problem II. 
	Namely, suppose we have $Z^M=(Z_1,\ldots,Z_{M}) \simiid \mu$ with $\mu$ some arbitrary distribution on $[k]$.
	The relevance to the Bernoulli sampling model comes from the following simple reduction: if $\mu$ is in fact the empirical
	distribution of colors, then given $\calN$, which corresponds a sample of size $M' \sim \Bino(k,p)$ from $\mu$
	without replacement, one can simulate an iid sample $Z_1,\ldots,Z_M$ with $M \approx (1-e^{-\bar{p}})k$.
	Hence, any result regarding estimating $\mu^{\downarrow}$ from $Z^M$ with $M=\Theta(k)$ implies a
	similar result about estimating $\mu^{\downarrow}$ from $\calN$ with $p = \Theta(1)$.

	\apxonly{Orig info:
	This corresponds to sampling $M$ balls from our urn with replacement, whereas our original formulation is in terms of
	sampling without replacement (and drawing $M' \sim \Bino(k, p)$ balls so). It turns out that from $M'\approx
	p k$ samples without replacement one can perfectly simulate a sample of $M\approx (1-e^{-\bar{p}})k$ samples
	with replacement. (TODO: add argument about forgetting the ball ids.)}
	
	We review several results regarding estimating $\mu^{\downarrow}$ from $Z^M$ when $\mu$ is general. The pioneering
	result~\cite{orlitsky2005convergence} only showed consistency, i.e. existence of estimator $\widehat {\mu^{\downarrow}}$
	such that
	$$ \Expect\|\widehat {\mu^{\downarrow}} - \mu^{\downarrow}\|_{\TV} \to 0 $$
	without convergence rate.	In a later draft~\cite{orlitsky2008estimating} (see also~\cite[Lemma
	3]{anevski2017estimating} for a short proof) it was shown that simply estimating $\mu^{\downarrow}$ by a sorted
	empirical distribution achieves 
	$$ \EE[\|\widehat {\mu^{\downarrow}} - \mu^{\downarrow}\|_{\infty}] = O(k^{-{1\over 2}} \log k)\,. $$

	A much more relevant result to us, however, is the one in~\cite{valiant2013estimating}. For any two
	$\pi^1,\pi^2$ they defined yet another distance:
\begin{equation}\label{eq:vvr_def}
		D(\pi^1,\pi^2) = \inf_{\nu} \EE[|\ln X_1 - \ln X_2|]\,,
\end{equation}	
	where the infimum is over all couplings of $X_1$ and $X_2$ distributed on $\mathbb{Z}_+$ as 
	$\PP[X_i = j] = j\pi^i_j$ for $i\in\{1,2\}, j \in [k]$. \apxonly{(I.e. $R(\pi^1,\pi^2) = W_1(p_1^@, p_2^@)$
	where $p_i^@(\ln j)=j\pi^i_j$.)} They have shown that when $M=a {k\over \log k}$ one can get 
	$$ \EE[D(\hat \pi, \pi)] \le O\left({1\over \sqrt{a}}\right),$$
	which, per~\cite{GValiant-email_2019}, also holds for $a= \Theta(\log k)$. In addition~\cite[Appendix
	B]{VV16_stoc} shows  $W_1(\pi^1,\pi^2) \le 2 D(\pi^1, \pi^2)$. 
	Indeed, let $\boldsymbol{\nu}(\cdot,\cdot)$ be the optimal coupling in~\eqref{eq:vvr_def}. Then define a coupling of $\pi^1$ to
	$\pi^2$ via
	$$ \tilde {\boldsymbol{\nu}}(j_1,j_2) = \begin{cases} {1\over \max(j_1,j_2)} \boldsymbol{\nu}(j_1, j_2), & j_1 \neq 0, j_2 \neq 0\\
					\sum_{j \ge j_1} \left({1\over j_1} - {1\over j}\right) \boldsymbol{\nu}(j_1,j), & j_2 = 0, j_1 > 0\\
					\sum_{j \ge j_2} \left({1\over j_2} - {1\over j}\right) \boldsymbol{\nu}(j,j_2), & j_1 = 0, j_2 > 0
				 \end{cases} $$
	and completing $j_1=j_2=0$ as required. Letting $(X_1,X_2) \sim \boldsymbol{\nu}$ and $(\tilde X_1, \tilde X_2) \sim \tilde
	{\boldsymbol{\nu}}$ we have that
	$$ \EE[|\tilde X_1 - \tilde X_2|] = 2 \EE[|\tilde X_1 - \tilde X_2|_+] = 2 \EE\left[{|X_1 - X_2|\over
	\max(X_1,X_2)}\right] \le 2 \EE[|\ln X_1 - \ln X_2|] = 2 D(\pi^1,\pi^2)\,.$$

	In all, putting everything together we have that Valiant and Valiant showed that there exists an estimator of
	$\mu^{\downarrow}$ from $M=\Theta(k)$ samples such that
	\begin{equation}\label{eq:jj_ach}
	\EE[\|\widehat {\mu^{\downarrow}} - \mu^{\downarrow}\|_{\TV}] = O\left({1\over \sqrt{\log k}}\right)\,.
	\end{equation}	
	In~\cite{han2018local} it was shown that this rate is minimax optimal over all distributions supported on $[k]$. Note, however, that since the lower bound in~\cite{han2018local} does not
	produce valid distributions on finite population (namely, $\mu$ with rational entries in ${1\over k} \integers$), it does imply that the rate of estimating $\pi$
	in $W_1$ is ${1\over \sqrt{\log k}}$, cf.~\eqref{eq:tv_w1}, is sharp.

	In all, we see that following the trailblazing work~\cite{orlitsky2005convergence} a number of works have established
	uniform convergence guarantees in various metrics. Relevant to us is that the best result available is 	$ \|\hat
	\pi - \pi\|_{\TV} \le O\left({1\over \sqrt{\log k}}\right)$, which can obtained by first simulating samples drawn with replacement based on those without replacements, then combining \eqref{eq:jj_ach} with \eqref{eq:mu_pi}. 
We show that this rate is suboptimal by a square root factor.

	\subsection{Open problems}
	For $1 \le q\le \infty$, let us define by $R_q(k)$  to be the minimax risk of estimating $\pi$ in the $\ell_q$-norm $\left(\sum_m
	|\pi_m - \hat\pi_m|^q\right)^{1\over q}$. Then in the linear regime of $p=\Theta(1)$, \prettyref{thm:main} shows that
		$$ \left({1\over \log k}\right)^{2-{1\over q}} \lesssim R_q(k) \lesssim {1\over \log k}\,,$$
		which is only tight for $q=1$. Our complex-analytic methods seem to be especially well suited for
		studying the case of $q=2$ and $q=\infty$, but we were not able to close the gap.
		The case of $\ell_\infty$ is of particularly interest as it concerns which individual profile is the hardest to estimate. Our result shows that for those colors that occur $m=\Theta(\log k)$ times, the corresponding $\pi_m$ is particularly difficult and cannot be estimated better than $\Omega(\frac{1}{(\log k)^2})$. It is unclear if this is the hardest case.

	Let us define by $R_{W_1}(k)$ to be the minimax risk of estimating $\pi$ in the 1-Wasserstein distance $W_1(\pi,\hat
	\pi)$. Given the equivalence~\eqref{eq:tv_w1}, estimate~\eqref{eq:jj_ach} and lower bound $W_1(\pi, \hat \pi)\ge
	\|\pi - \hat \pi\|_{\TV}$ we get
		$$ {1\over \log k}\lesssim R_{W_1}(k) \lesssim {1\over \sqrt{\log k}}\,.$$
	Due to $W_1$ being the $L_1$-distance between the CDFs, the minimax $W_1$ risk are also amenable to complex-analytic techniques, but so far resisted our attempts. 
	An alternative approach is to generalize the $W_1$-lower bound construction of \cite{han2018local}; however, as observed in previous work in the distinct elements problem \cite{GV-thesis,WY16-distinct} such moment-based construct is difficult to extend to finite population.

\section{Impossibility of learning the empirical distribution}
	\label{app:mu}

In this section we show that unless we observe all but a vanishing fraction of the urn, it is impossible to estimate the empirical distribution of the colors consistently.
To this end, consider a $k$-ball urn and let $\mu$ denote the empirical distribution of the colors, with $\mu(j) = \frac{\theta_j}{k} , j \in [k]$.
Compared to the profile $\pi$ which is a distribution on $\integers_+$, here $\mu$ is a probability measure on the set of colors $[k]$. 
Similar to \prettyref{eq:Rk}, we define the minimax TV risk for estimating $\mu$:
	$$
	\tilde R(k)=\inf_{\hat{\mu}}\sup_{\mu}\EE[\|\mu-\hat\mu\|_{\TV}].
	$$
The following theorem shows that whenever the sampling ratio $p$ is bounded away from one, it is impossible to estimate $\mu$ consistently.
This observation agrees with the typical behavior in high-dimensional estimation that, absence any structural assumptions, the sample size need to exceed the number of parameters to achieve consistency.

\begin{theorem}
\label{thm:Rk-mu}
	\[
	\tilde R(k) \geq \frac{k-1}{4k} h^{-1}\pth{1-{p} - \frac{\log_2(k+1)}{k-1}} 
	\]
	where $h:[0,1]\to[0,1]$ given by $h(x) = -x \log_2 x-(1-x) \log_2(1-x)$ is the binary entropy function,
	and $h^{-1}$ is its inverse on $[0,\frac{1}{2}]$.
	Consequently, for any fixed $p<1$, $\tilde R(k) = \Omega(1)$.
\end{theorem}
\begin{proof}
The proof follows the mutual information method that compares the amount of information data provides and the minimum amount of information needed to reconstruct the parameters up to a certain accuracy.
Consider the following Bayesian setting of a $k$-ball urn, where $\theta_j\iiddistr	Ber(1/2)$ for $j=1,\ldots,k-1$ and $\theta_k = k-\sum_{j<k} \theta_j$. 
	In other words, each of the first $k-1$ colors either is absent or appear exactly once with equal probability.
	Then for $j\in[k-1]$, the observed $X_j$ is simply the erased version of $\theta_j$ with erasure probability $\bar p$.
	Thus the mutual information (in bits) between the parameters 
	$\theta=(\theta_j: j\in[k-1])$ and the observations $X=(X_j: j\in[k])$
	can be upper bounded as follows:
	\begin{align*}
	 I(\theta; X)
	= & ~ \underbrace{I(\theta; X_{1},\ldots,X_{k-1})}_{=(k-1)p} + \underbrace{I(\theta; X_{k}|X_{1},\ldots,X_{k-1})}_{\leq H(X_{k}) \leq \log_2(1+k)}
	\end{align*}
	where the inequality follows from the fact that $X_{k}$ takes at most $k$ values.
	On the other hand, suppose there exists $\hat\mu=\hat\mu(X)$,
	such that 
	$\Expect[\|\mu-\hat\mu\|_{\TV}]\leq \epsilon$.
	Define 
	$\hat \theta_j = \indc{\hat \mu_j > \frac{1}{2k}}$ for $j\in[k-1]$.
	Then 
	$2\|\mu-\hat\mu\| \geq \sum_{i=1}^{k-1} \|\mu_j-\hat\mu_j\| \geq \frac{1}{2k} \sum_{i=1}^{k-1} \indc{\theta_j \neq \hat \theta_j}\|$.
	Thus $\hat\theta$ are close to $\theta$ in Hamming distance: 
	$\sum_{i=1}^{k-1} \pprob{\theta_j \neq \hat \theta_j} \leq 4\epsilon k$.
	By the rate-distortion function of Bernoulli distribution \cite[Chap.~10]{cover},
	their mutual information must be lower bounded by
	$$
	I(\theta; \hat\theta) 	\ge (k-1) \pth{ 1-h\pth{\frac{4\epsilon k}{k-1}}}.$$
	Combined with the data processing inequality $I(\theta; X) \geq I(\theta; \hat\theta)$, the last two displays imply that
	$\epsilon \ge \frac{k-1}{4k} h^{-1}(\bar{p} - \frac{\log_2(k+1)}{k-1}) $ which concludes the proof. 
\end{proof} 

\section{Proof of \prettyref{thm:risk-deltatv}}
	\label{app:risk-deltatv}

	\begin{proof}
		We first prove the upper bound by analyzing the minimum distance estimator \prettyref{eq:pihat}. 
		Let $\pi\in\Pi_k\subset\Pi$ denote the true profile. Denote the distribution $\nu\triangleq \pi P$. 
		As outlined in \prettyref{sec:mind} and in view of \prettyref{eq:dagger}, the key step is to show that $\hat \nu$ is concentrated around $\nu$ in terms of total variation. To this end, observe that for $m\geq 1$, we have $\Expect[\hat \nu_m]=\nu_m$ from \prettyref{eq:unbiased}. Furthermore, 
		\begin{align}\label{eq:m4}
		k\cdot \Var[\hat\nu_m]={1\over k} \Var[Y_m] = {1\over k} \sum_{j\in \calX} \Var[1\{X_j = m\}] \le {1\over k} \sum_{j\in \calX} \PP[X_j=m] = (\pi P)_m=\nu_m\,. 
		\end{align}
		 Thus
		$\expect{|\hat{\nu}_m-\nu_m|}\leq \sqrt{\nu_m/k}$.
		Summing over $m$ we get
		\begin{align}
		\Expect[\|\hat\nu-\nu\|_{\TV}]\leq
		\expect{\sum_{m=1}^{k}|\hat{\nu}_m-\nu_m|}
		& \leq \frac{1}{\sqrt{k}}\sum_{m=1}^k \sqrt{\nu_m}\nonumber \\
		& \stepa{\le} \frac{1}{\sqrt{k}} \left(\sum_{m=1}^k m \nu_m \right)^{1/2}\left(\sum_{m=1}^k \frac 1m \right)^{1/2} \nonumber \\
		&\stepb{\le}  O\pth{{\sqrt{\log k\over k}}} ,\label{eq:m16}
		\end{align}
		where (a) follows from Cauchy-Schwarz; (b) follows as follows:
		if we denote $U_1\sim \pi$ and $U_2|U_1 \sim \Binom(U_1,p)$, then $U_2\sim \nu$  and hence $\Expect[U_2] = p \Expect[U_1] \leq p$ thanks to the mean constraint on $\pi\in \Pi$.
		Next we show that
		\begin{equation}
		\prob{|\|\nu - \hat \nu\|_{\TV} - \Expect\|\nu - \hat \nu\|_{\TV}| \geq \epsilon } \leq \exp(-C_0 k \epsilon^2)
		\label{eq:tv-concentration}
		\end{equation}
		for some absolute constant $C_0$, all $\epsilon>0$ and $k$ large. For that we aim to show that $\|\nu - \hat \nu\|_{\TV}$ satisfies the bounded difference property and then apply McDiarmid's inequality. 
		Let $x_1,\dots,x_{\tilde k}$ be the distinct colors present in the urn with $\tilde k\leq k$. 
		Denote $\|\nu - \hat \nu\|_{\TV} = d(N_{x_1},\dots, N_{x_{\tilde k}})$ for some function $d$.
		Then $d$ satisfies the following:
		for any $i \in [\tilde k]$ and any $n_1,\ldots,n_{\tilde k}$ with $n_i' \neq n_i$, we have
		\begin{align}
		&\left|d(n_1,\dots,n_{i-1},n_i,n_{i+1},\dots,n_{\tilde k})- d(n_1,\dots,n_{i-1},n_i',n_{i+1},\dots,n_{\tilde k})\right|\nonumber \\
		\leq & \frac 12\left||\nu_{n_i}-\hat \nu_{n_i}|+|\nu_{n_i'}-\hat \nu_{n_i'}|-\left|\nu_{n_i}-\pth{\hat{\nu}_{n_i}-\frac{1}{k}}\right|-\left|\nu_{n_i'}-\pth{\hat{\nu}_{n_i'}+\frac{1}{k}}\right| \right|\label{eq:m31}\\
		\leq &\frac {1}{k}.\nonumber
		\end{align}
		Furthermore, $(N_{x_1},\dots, N_{x_{\tilde k}})$ are independent. 
		Then the desired exponential bound in \prettyref{eq:tv-concentration} follows from McDiarmid's inequality. 
		
		Combining \eqref{eq:m16} and \eqref{eq:tv-concentration} we get
		\begin{equation}
		\prob{\|\nu - \hat \nu\|_{\TV} \geq \sqrt{\frac{C_1 \log k}{k}} } \leq k^{-1}
		\label{eq:tv-concentration1}
		\end{equation}
		for some absolute constant $C_1$. Then taking expectations on both sides of \prettyref{eq:dagger}, for sufficiently large $k$ we get
		\begin{align*}
		\Expect \|\hat{\pi}-\pi \|_{\TV}
		\leq & ~ \Expect[\deltaTV(2 \|\pi P - \hat \nu\|_{\TV})] \\
		\stepa{\leq} & ~ \deltaTV\pth{\sqrt{\frac{C_1 \log k}{k}}} + k^{-1} \\
		\stepb{\leq} & ~ 2\deltaTV\pth{\sqrt{\frac{C_1 \log k}{k}}},
		\end{align*}
		where (a) follows from \prettyref{eq:tv-concentration1} and $\deltaTV \leq 1$, (b) follows from the universal fact that $\deltaTV(t) \geq t$ (\prettyref{rmk:reverseDPI})	and $\delta_{\TV}(t)$ is increasing in $t$. This yields the desired upper bound on $R(k)$.

		To show the lower bound, consider any bounded function $T: \integers_+ \to [-1,1]$. Then for distribution $\pi$ on $\integers_+$, define the linear functional $T_{\pi}$:
		\begin{align*}
		T_{\pi}=\sum_{m} \pi_m T(m).
		\end{align*}
		Note that $2\|\hat \pi - \pi\|_{\TV} = \sup_{T} |T_{\hat \pi}-T_{\pi}|$ for any estimator $\hat \pi$ of $\pi$. Hence the minimax TV risk of estimating $\pi$ can be lower bounded by that of estimating $T$
		\begin{align*}
		R(k)\geq \frac 12 R_T(k), \quad R_T(k) \triangleq  \inf\sup \EE\qth{|\hat T- T_{\pi}|}. 
		\end{align*}
		where the estimator $\hat T$ depends on $(X_j:j\in \calX)$ and the supremum is again over all $k$-ball urns. 
		We are now in position to apply \cite[Theorem 8]{PW18-dual2} (with $\Theta=\integers_+$, $c(\theta)=\theta$,
		and $K_V=1$) to obtain\footnote{The result of \cite[Theorem 8]{PW18-dual2} is stated in terms of the
		$\chi^2$-divergence. The TV version follows by applying~\cite[Proposition 1]{PW18-dual2} to lower bound 
		$\delta_{\chi^2}(t)$ via $\deltaTV(t)$.}
		\[
		R_T(k)  \geq \frac 1{72}\deltaTV\pth{\frac 1{6k}}-{C_2\over \sqrt k}
		\]
		where
		\begin{align}
		\deltaTV(t,T)= \sup\{|T_{\pi'}-T_{\pi}|:  \TV(\pi'P,\pi P)\le t, \pi, \pi'\in\Pi\} \label{eq:deltv_def}
		\end{align}
		Finally optimizing over $T$ observing that $\deltaTV(t)=\sup_{T}\deltaTV(t,T)$ for every $t>0$ yields the result.
		\end{proof}

\section{Proofs of technical lemmas}
\label{app:proofs}

\begin{proof}[\prettyref{lmm:delta*}]
	We prove the lemma by showing how a feasible solution of one of the programs can be utilized to get a feasible solution of the other one, and vice-versa. Let us start with the second inequality. Given any pair ($\pi$,$\pi'$) feasible for $\deltaTV(t)$, choose $\Delta=(\pi-\pi')/2$. We get
	\begin{align*}
	\sum_m m|\Delta_m|=\frac{1}{2}\sum_m m|\pi_m-\pi_m'|
	\leq \frac{1}{2}\sum_m m(\pi_m+\pi_m')
	\leq 1.
	\end{align*}
	The relation $\|\Delta P\|_1\leq t$ follows directly from $\|\pi P-\pi'P\|_{\TV}\leq t$. This shows  $\Delta$ is feasible for $\delta_*(t)$ with $\|\Delta\|_1=\|\pi-\pi'\|_{\TV}$. This proves the second inequality in \prettyref{lmm:delta*}.
	
	The first inequality is proven next. Take any non-zero feasible solution $\tilde\Delta$ to $\delta_*(t)$ (which exists
	because we can always choose $\tilde\Delta=0$). Next, suppose that $\epsilon \eqdef \sum_m \tilde \Delta_m \neq 0$. 
	Then, let us define $\Delta_j = \tilde \Delta_j$ for $j\ge 1$ and $\Delta_0 =
	\tilde\Delta_0-\epsilon$. It is clear that 
	\begin{equation}\label{eq:sumd_1}
		\sum_j \Delta_j= 0
	\end{equation}	
	Furthermore, since $\langle \tilde \Delta P,
	\mathbf{1}\rangle = \langle \tilde \Delta, \mathbf{1} \rangle = \epsilon$ we conclude that $|\epsilon| \le
	\|\Delta P\|_1 \le t$. Therefore,
	\begin{equation}\label{eq:sumd_2}
		\sum_j |\Delta_j| \ge \sum_j |\tilde \Delta_j| - t\,. 
	\end{equation}	
	Finally, because $\|rP\|_1 \le \|r\|_1$ we also have from triangle inequality
	\begin{equation}\label{eq:sumd_3}
		\|\Delta P\|_1 \le t + |\epsilon| \le 2t\,.
	\end{equation}	

	Next we define $\Delta^+=\max(\Delta,0)$, $\Delta^-=\max(-\Delta,0)$, where max is defined coordinate wise. We 
	choose $\{\pi_m\}_{m=0}^\infty$ and $\{\pi_m'\}_{m=0}^{\infty}$ as
	\begin{align*}
	\pi_0=1-\sum_{j=1}^{\infty} \Delta^+_j,&\quad \pi_0'=1-\sum_{j=1}^{\infty} \Delta^-_j,\\
	\pi_m=\Delta^+_m, &\quad \pi_m'=\Delta^-_m,\quad m\ge 1
	\end{align*}
	Note that under constraints on $\Delta$, we have $\pi,\pi'\in \Pi$. Indeed, $\sum_{m\ge 1} |\Delta_m| \le \sum_m m
	|\Delta_m| = \sum_m m |\tilde \Delta_m| \le 1$ and thus $\pi_0,\pi_0'\ge 0$. Furthermore, since $|\Delta_m|=\Delta_m^+ + \Delta_m^-$ we have
	$
	\sum_m m(\Delta_m^+ + \Delta_m^-)\leq 1
	$
	which implies 
	$
	\sum_m m(\pi_m + \pi_m')\leq 1.
	$ 
	This proves $\sum_m m\pi_m\leq 1$ and $\sum_m m\pi_m'\leq 1$. Next, observe that $\pi_0 - \pi_0' = \Delta_0$
	due to~\eqref{eq:sumd_1} and thus $\pi-\pi' = \Delta$. From~\eqref{eq:sumd_3} we conclude that $\|(\pi -
	\pi')P\|_{\TV} \le t$ and hence $(\pi,\pi')$ is a feasible pair for $\delta_{\TV}(t)$. And thus
	via~\eqref{eq:sumd_2} we obtain
		$$ \deltaTV(t) \ge {1\over 2} \left(\delta_*(t)-t\right)\,.$$
\end{proof}

\begin{proof}[Proof of \prettyref{lmm:lag}]
		In view of~\eqref{eq:h7} and~\eqref{eq:Plancherel} the proof of~\eqref{eq:h8} is straightforward but
		delicate. To simplify analysis we will assume $\beta \to \infty$ and denote by $o(1)$ the terms
		vanishing with $\beta$.

		For $m\in \pth{\beta,\frac {3\beta}{2}}$ we define $\phi_m=\arccos \sqrt{\beta/(2m)}$ and $\theta_m=F(\phi_m)$ where $F(\phi)=\sin(2\phi)-2\phi$. Here $\phi_m\in(\arccos (1/2),\arccos (1/3))$ and hence is bounded away from both 0 and $\pi/2$ for all $m$ in the above range. Then using \eqref{eq:Plancherel} with $x=2\beta$, we get that there exist absolute constants $\beta_0,C_7$ such that for all $\beta \geq \beta_0$,
		\begin{align}
		|L_m^{(-1)}(2\beta)|&\geq {e^{\beta}\over \sqrt{\pi}\pth{1-\frac 13}^{1/4}}(2\beta)^{1/4}m^{-3/4}\sth{\left| \sin\pth{m\theta_{m}+{3\pi\over 4}}\right|-C_7\beta^{-1}}\nonumber \\
		&\geq {2e^{\beta}\over \sqrt{\pi}(2/3)^{1/4}3^{3/4}}\beta^{-1/2}\sth{\left|\sin\pth{m\theta_{m}+{3\pi\over 4}}\right|-C_7\beta^{-1}}\nonumber \\
		&\geq {e^{\beta}\beta^{-1/2}\over 2}\sth{\left|\sin\pth{m\theta_{m}+{3\pi\over 4}}\right|-C_7\beta^{-1}}.\label{eq:m34}
		\end{align} 
		Now we consider any two consecutive integers $m$ and $m+1$ in $\pth{\beta,\frac {3\beta}{2}}$. Using \eqref{eq:m34} we get
		\begin{align}
		&|L_m^{(-1)}(2\beta)|+|L_{m+1}^{(-1)}(2\beta)|\nonumber \\
		\geq &{e^{\beta}\beta^{-1/2}\over 2}\sth{\left|\sin\pth{m\theta_{m}+{3\pi\over 4}}\right|+\left|\sin\pth{\pth{m+1}\theta_{m+1}+{3\pi\over 4}}\right|-2C_7\beta^{-1}}\label{eq:m23}.
		\end{align}
		The phase difference between the two sine terms comes out to be $m(\theta_{m}-\theta_{m+1})-\theta_m$. Using the formula $\theta_m=F(\phi_m)$, we get
		\begin{align}
		m(\theta_m-\theta_{m+1})=m(\phi_m-\phi_{m+1})  {F(\phi_m)-F(\phi_{m+1})\over \phi_m-\phi_{m+1}}.\label{eq:m35}
		\end{align}
		We will show that the above is bounded away from 0 as $m$ goes to infinity. We first consider the term $m(\phi_m-\phi_{m+1})$. Using $\fracd{}{x}\arccos \sqrt{x}=-\frac 12 {1\over \sqrt{x(1-x)}}$ we deduce that 
		\begin{align*}
		m\pth{\phi_{m}-\phi_{m+1}}&= m\pth{\arccos{\sqrt{{\beta\over 2m}}}-\arccos{\sqrt{{\beta\over 2m+2}}}} \\
		&= m\pth{\arccos{\sqrt{{\beta/2m}}}-\arccos{\sqrt{{\beta/2m}-{{\beta/2m}\over m+1}}}} \\
		&={\beta\over 2m}\cdot {m \over m+1}\cdot{\arccos{\sqrt{{\beta/2m}}}-\arccos{\sqrt{{\beta/2m}-{{\beta/2m}\over(m+1)}}}\over {{\beta/2m}\over(m+1)}} \\
		&=-\frac 12\sqrt{{\beta/2m}\over {1-{\beta/2m}}}+o(1)
		\end{align*}
		where the $o(1)$ term goes to 0 as $m,\beta$ tends to infinity with $\frac{\beta}{2m}\in \pth{\frac 13,\frac 12}$. In view of \eqref{eq:m35} using $F'(\phi)=2\cos(2\phi)-2$ and $\cos^2(\phi_m)=\frac \beta{2m}$
		we get
		\begin{align}
		m(\theta_{m}-\theta_{m+1})&=-\frac 12\sqrt{{\beta/2m}\over {1-{\beta/2m}}}  F'(\phi_m)+o(1)\nonumber \\
		&=-2\sqrt{{\beta/2m}\over {1-{\beta/2m}}}\pth{{\beta\over 2m}-1}+o(1)\nonumber \\
		&={2\sqrt{{\beta\over2m}\pth{1-{\beta\over2m}}}}+o(1)\label{eq:Phase diff}
		\end{align}
		with the same last conditions on $m,\beta$. As $\frac \beta{2m}\in\pth{\frac 13,\frac 12}$ the above quantity is bounded away from 0. Also \eqref{eq:Phase diff} implies that $\theta_{m+1}$ can be approximated as $\theta_m+o(1)$.
		As we have
		\begin{align*}
		\theta_m&=\sin(2\phi_m)-2\phi_m\\
		&=2\sin \phi_m \cos \phi_m -2\phi_m\\
		&=2\sqrt{{\beta\over2m}\pth{1-{\beta\over2m}}}-2\phi_m
		\end{align*}
		continuing \eqref{eq:m23} and using \eqref{eq:Phase diff} we get 
		\begin{align}
		&|L_m^{(-1)}(2\beta)|+|L_{m+1}^{(-1)}(2\beta)|\nonumber \\
		\geq &{e^{\beta}\beta^{-1/2}\over 2}\sth{\left|\sin\pth{m\theta_{m}+{3\pi\over 4}}\right|+\left|\sin\pth{m\theta_{m}+{3\pi\over 4}+\theta_{m}-2\sqrt{{\beta\over 2m}\pth{1-{\beta\over 2m}}}}\right|+o(1)}\nonumber \\
		= &{e^{\beta}\beta^{-1/2}\over 2}\sth{\left|\sin\pth{m\theta_{m}+{3\pi\over 4}}\right|+\left|\sin\pth{m\theta_{m}+{3\pi\over 4}-2\phi_m}\right|+o(1)}.\label{eq:m26}				
		\end{align}
		Now we note that for any real number $a\in (0,\pi)$ the function $s(x)\triangleq |\sin (x)|+|\sin(x-a)|$ has period $\pi$ and is piecewise concave on the intervals $(0,a)$ and $(a,\pi)$. As $s(0)=s(a)=s(\pi)=\sin (a)$ we get
		\begin{align*}
		\inf_{j}\sth{|\sin (x)|+|\sin(x-a)|}=\sin(a).
		\end{align*}
		In view of the above, continuing \eqref{eq:m26} we get
		\begin{align*}
		{|L_m^{(-1)}(2\beta)|+|L_{m+1}^{(-1)}(2\beta)|}&\geq{e^{\beta}\beta^{-1/2}\over 2}\sth{\sin\pth{2\phi_m}+o(1)} \\
		&= {e^{\beta}\beta^{-1/2}\over 2}\sth{2\sqrt{{\beta\over 2m}\pth{1-{\beta\over 2m}}}+o(1)} \\
		&\geq {e^{\beta}\beta^{-1/2}\over 2}\pth{\frac {2\sqrt{2}}3+o(1)}
		\end{align*}
		for any $m\in \pth{\beta,\frac{3\beta}{2}}$. In view of~\eqref{eq:h7} this implies~\eqref{eq:h8}.	
\end{proof}

\end{document}